\documentclass[12pt]{article}
\usepackage{amsmath,amsthm, amscd,amssymb,amsfonts}
\usepackage{graphicx}
\usepackage{verbatim}
\usepackage{mathtools}
\usepackage{datetime}
\usepackage{tikz}

\usepackage{url}
\usepackage[mathlines]{lineno}
\usepackage{color}
\usepackage{enumerate}

\definecolor{red}{rgb}{1,0,0}

\definecolor{purp}{rgb}{0.55,0.15,0.51}

\usepackage{tikz}
\usepackage{rotating}

\usepackage{graphicx}
\usepackage{caption}
\usepackage{subcaption}
\usepackage{hyperref}

\newenvironment{pfct}[1]{\noindent{\bf Proof of #1.\,}}{\hfill$\Box$ \\}

\newtheorem{thm}{Theorem}
\newtheorem{conj}[thm]{Conjecture}
\newtheorem{lem}[thm]{Lemma}

\newtheorem{defn}[thm]{Definition}

\newcommand{\abs}[1]{\left\lvert{#1}\right\rvert}
\newcommand{\floor}[1]{\left\lfloor{#1}\right\rfloor}

\newcommand{\spine}{{\rm spine}}

\newcommand{\C}{\mathcal{C}}

\newcommand{\La}{{\rm La}}
\newcommand{\B}{\mathcal{B}}

\newcommand{\F}{{\mathcal F}}

\newdateformat{mydate}{\monthname \vspace{2px} \twodigit{\THEDAY}, \vspace{2px}\THEYEAR}

\usepackage{authblk}

\title{A simple discharging method for forbidden subposet problems}
\date{\mydate\today}

\author[1]{Ryan R. Martin\footnote{Research supported in part by a Simons Foundation grant (\# 353292)}\thanks{Research supported in part by NSF-DMS Grants 1604458, 1604773, 1604697, and 1603823.}}
\author[2]{Abhishek Methuku\thanks{Research supported in part by the Hungarian National Research, Development and Innovation Office -- NKFIH under the grant K116769.}\thanks{Research supported in part by a generous grant from the Combinatorics Foundation.}}
\author[3]{Andrew Uzzell\protect\footnotemark[2]\thanks{Research supported in part by a generous grant from the Institute for Mathematics and its Applications.}}
\author[1]{Shanise Walker\protect\footnotemark[2]\protect\footnotemark[5]}

\affil[1]{\footnotesize Department of Mathematics, Iowa State University, Ames, IA, USA. \texttt{\{rymartin,shanise1\}@iastate.edu}}
\affil[2]{\footnotesize Department of Mathematics, Central European University, Budapest, Hungary. \texttt{abhishekmethuku@gmail.com}}
\affil[3]{\footnotesize Department of Mathematics and Statistics, Grinnell College, Grinnell, IA, USA. \texttt{uzzellan@grinnell.edu}}

\begin{document}
\maketitle

\begin{abstract}
The poset $Y_{k, 2}$ consists of $k+2$ distinct elements  $x_1$, $x_2$, \dots,~$x_{k}$, $y_1$,~$y_2$, such that $x_1 \le x_2 \le \dots \le x_{k} \le y_1$,~$y_2$. The poset $Y'_{k, 2}$ is the dual poset of $Y_{k, 2}$. The sum of the $k$ largest binomial coefficients of order $n$ is denoted by $\Sigma(n,k)$.
Let $\La^{\sharp}(n,\{Y_{k, 2}, Y'_{k, 2}\})$ be the size of the largest family $\F \subset 2^{[n]}$ that contains neither $Y_{k,2}$ nor $Y'_{k,2}$ as an induced subposet. Methuku and Tompkins proved that $\La^{\sharp}(n, \{Y_{2,2}, Y'_{2,2}\}) = \Sigma(n,2)$ for $n \ge 3$ and they conjectured the generalization that if $k \ge 2$ is an integer and $n \ge k+1$, then $\La^{\sharp}(n, \{Y_{k,2}, Y'_{k,2}\}) = \Sigma(n,k)$. On the other hand, it is known that $\La^{\sharp}(n, Y_{k,2})$ and $\La^{\sharp}(n, Y'_{k,2})$ are both strictly more than $\Sigma(n,k)$. In this paper, we introduce a simple discharging approach and prove this conjecture.  

\end{abstract}

\section{Introduction}
	
The \emph{$n$-dimensional Boolean lattice}, denoted $\B_n$, is the partially ordered set (poset) $(2^{[n]}, \subseteq)$, where $[n]=\{1,\ldots,n\}$.  For any $0 \le i \le n$, let $\binom{[n]}{i} := \{A \subseteq [n] : \abs{A} =i\}$ denote the \emph{$i$th level} of the Boolean lattice.  Let $P$ be a finite poset and $\F$ be a family of subsets of $[n]$. We say that $P$ is contained in $\F$  as a {\em weak }\! subposet if there is an injection $\alpha : P \rightarrow \mathcal{F}$ satisfying $x_1 <_p x_2 \Longrightarrow \alpha(x_1)\subset \alpha(x_2)$ for all $x_1$,~$x_2\in P$.
$\F$ is called \emph{$P$-free} if $P$ is not contained in $\F$ as a weak subposet.  
We define the corresponding extremal function to be
$\La(n,P) := \max \{ \abs{\F} : \mathcal{F} ~\text{is $P$-free}\}$. Analogously, if $P$, $Q$ are two posets then, let $\La(n, \{P, Q\}) := \max \{ \abs{\F} : \mathcal{F} ~\text{is $P$-free and $Q$-free}\}$.

The linearly ordered poset on $k$ elements, $a_1 < a_2 < \ldots < a_k$, is called a {\em chain of length $k$}, and is denoted by $P_k$. 
Using this notation the well-known theorem of Sperner~\cite{Sper} can be stated as $\La(n, P_2) = {\binom {n} {\lfloor n/2 \rfloor}}$. Let us denote the sum of the $k$ largest binomial coefficients of order $n$ by $\Sigma(n,k)$.  Erd\H{o}s \cite{Erd} extended Sperner's theorem by showing that $\La(n, P_{k}) = \Sigma(n,k-1)$ with equality if and only if the family is union of $k-1$ largest levels of the Boolean lattice. Notice that, since any poset $P$ is a weak subposet of a chain of length $|P|$, Erd\H{o}s's theorem implies that
$\La(n,P) \le (|P|-1){\binom {n} {\lfloor n/2 \rfloor}}=O \left( \binom {n} {\lfloor n/2 \rfloor} \right ).$ Later many authors, including Katona and Tarj\'{a}n~\cite{KatTar}, Griggs and Lu~\cite{GLu}, and Griggs, Li, and Lu~\cite{GLL} studied various other posets (see the recent survey by Griggs and Li \cite{GriggsLi} for an excellent survey of all the posets that have been studied). Let $h(P)$ denote the height (maximum length of a chain) of~$P$. One of the first general results is due to Bukh who showed that if $T$ is a finite poset whose Hasse diagram is a tree of height $h(T) \ge 2$, then $\La(n,T)= (h(T)-1+O(1/n)){\binom {n} {\lfloor n/2 \rfloor}}$. The most notorious poset for which the asymptotic value of the extremal function is still unknown is the {\em diamond} $D_2$, the poset on $4$ elements with the relations $a<b,c<d$ where $b$ and $c$ are incomparable. The best known bound is $(2.20711+o(1)) {\binom {n} {\lfloor n/2 \rfloor}}$, due to Gr\'osz, Methuku, and Tompkins~\cite{GMT2}.

We say that $P$ is contained in $\F$ as an {\em induced subposet} if and only if there is an injection~$\alpha : P \rightarrow \mathcal{F}$ satisfying $x_1 <_p x_2 \Longleftrightarrow \alpha(x_1)\subset \alpha(x_2)$ for all $x_1$,~$x_2\in P$.
We say that $\F$ is {\em induced-$P$-free} if $P$ is not contained in $\F$ as an induced subposet. 
We define the corresponding extremal function as
$\La^{\sharp}(n,P) := \max \{ \abs{\F} : \mathcal{F} ~\text{is induced $P$-free}\}.$ Analogously, if $P$, $Q$ are two posets then let $ \La^{\sharp}(n,\{P, Q\}) := \max \{ \abs{\F} : \mathcal{F} ~\text{is induced $P$-free and induced $Q$-free}\}$.

Despite the considerable progress that has been made on forbidden weak subposets, little is known about forbidden induced subposets (except for $P_k$, where the weak and induced containment are equivalent).
The first results of this type are due to Carroll and Katona~\cite{CarK}, and due to Katona~\cite{Katona_survey}, showing
$\La^{\sharp}(n,V_{r})=\left(1+o(1)\right){\binom {n} {\lfloor n/2 \rfloor}}$ where $V_{r}$ is the $r$-fork poset ($x \le y_i$ for all $1 \le i \le r$). 
Boehnlein and Jiang~\cite{EdTao} generalized this by extending Bukh's result to induced containment of tree-shaped posets, $T$, proving
$\La^{\sharp}(n,T)= (h(T)-1+o(1)){\binom {n} {\lfloor n/2 \rfloor}}$. Only recently, Methuku and P\'alv\"olgyi \cite{MPal} showed that for every poset $P$, there is a constant $c_P$ depending only on $P$ such that $\La^{\sharp}(n,P) \le c_P{\binom {n} {\lfloor n/2 \rfloor}}$.

Even fewer exact results are known for forbidden induced subposets, which is the topic of this paper.   Katona and Tarj\'an \cite{KatTar} proved that $\La(n,\{V,\Lambda\}) = \La^{\sharp}\left(n,\{V,\Lambda\}\right) =  2 \binom{n-1}{\floor{\frac{n-1}{2}}}$, where $V$ and $\Lambda$ are the $2$-fork and its dual, the $2$-brush, respectively.  



Now we formally define the posets considered in this paper.

\begin{defn}
Let $k$,~$r \ge 2$ be integers. The $r$-fork with a $k$-shaft poset consists of $k+r$ elements $x_1$,$x_2,\dots,x_k$,$y_1$, $y_2, \ldots, y_{r-1}, y_r$ with $x_1 \le x_2 \le \dots \le x_k$ and $x_k \le y_i$ for all $1 \le i \le r$, and is denoted by  $Y_{k, r}$. Let $Y'_{k, r}$ denote the reversed poset of  $Y_{k, r}$, also called the dual poset of $Y_{k,r}$. 

For simplicity, we will write $Y_{k}$ and $Y'_{k}$ instead of $Y_{k, 2}$ and $Y'_{k, 2}$ respectively.
\end{defn}

The first result about $Y_{k,r}$ was due to Thanh \cite{Tran} who showed that 
$\La(n, Y_{k,r}) = (k + o(1)){\binom {n} {\lfloor n/2 \rfloor}}$. The lower order term in his upper bound was improved by De Bonis and Katona \cite{Krfork}.  Thanh also gave a construction showing that $\La(n, Y_{k,r}) > \Sigma(n, k)$. Methuku and Tompkins \cite{MCasey} showed that if one forbids both $Y_k$ and $Y'_k$, then an exact result can be obtained: $\La(n, \{Y_k, Y'_k\}) = \Sigma(n, k)$.

Using a cycle decomposition method, they also showed the following exact result for induced posets.

\begin{thm}[Methuku--Tompkins~\cite{MCasey}]  \label{thm:zero}
	If $n \ge 3$, then $\La^{\sharp}(n, \{Y_{2}, Y'_{2}\}) = \Sigma(n,2)$.
\end{thm} 

Theorem \ref{thm:zero} strengthens the result of De Bonis, Katona, and Swanepoel \cite{DKS} stating that $\La(n,B) = \Sigma(n,2)$ where $B$ is the butterfly poset which consists of $4$ elements $a$, $b$, $c$,~$d$ with $a$,~$b \le c$,~$d$. Indeed if a family does not contain the butterfly as a subposet, then it contains neither $Y_{2}$ nor  $Y'_{2}$ as an induced subposet. However, a family might contain neither an induced $Y_2$ nor an induced $Y_2'$ while still containing a butterfly.

In Section~\ref{results}, we establish the following generalization of Theorem~\ref{thm:zero} by proving a conjecture from~\cite{MCasey}. 

\begin{thm}\label{thm:conjresult}
	If $k \ge 2$ is an integer and $n \ge k+1$, then $\La^{\sharp}(n, \{Y_k, Y'_k\}) = \Sigma(n,k)$.
\end{thm} 

Note that forbidding only one of $Y_k$ and $Y'_k$ is not enough to obtain an exact result. Indeed, again by Thanh's construction \cite{Tran} we have $\La^{\sharp}(n, Y_k) > \Sigma(n,k)$ and $\La^{\sharp}(n, Y'_k) > \Sigma(n,k)$.

We further obtain the following LYM-type inequality if we assume $\varnothing$ and $[n]$ are not in our family.

\begin{thm}\label{thm:firstLYM}
Let $k \ge 2$ be an integer and $n \ge k+1$. If $\F \subset 2^{[n]}$ contains neither $Y_k$ nor $Y'_k$ as an induced subposet and $\varnothing$,~$[n] \notin \F$, then
	\begin{displaymath}
	\sum_{F\in \F} \binom{n}{\abs{F}}^{-1} \le k.
	\end{displaymath}
	In particular, $\abs{\F}\leq\Sigma(n,k)$.
\end{thm} 

\section{Preliminaries}\label{prelim}

The following terminology will be used to prove Theorems~\ref{thm:conjresult} and \ref{thm:firstLYM}. Let $\F$ be a family of subsets of $[n]$ which is induced $Y_k$-free and induced $Y'_k$-free. For sets $U$,~$V \subseteq [n]$, let the {\em interval $[U,V]$} denote the Boolean lattice induced by the collection of all sets that contain $U$ and are contained in $V$.   A chain $C$ where  $C=\{A_0, \dots, A_n\}$ and $\varnothing=A_0\subset A_1\subset\cdots \subset A_n=[n]$ is called a \emph{full chain} or a \emph{maximal chain}.

A \emph{spine} $S$ is a chain $A_1\subset A_2\subset \cdots \subset A_{\ell}$ such that $\abs{A_{i+1} \setminus A_{i}} = 1$  for $1\leq i\leq \ell-1$ where there are exactly $k-1$ members of $\F$ in $\{ A_1, \dots, A_{\ell}\}$ and where $A_1$,~$A_{\ell}\in\F$. Note that a spine may contain elements not from $\F$.

Let $\mathcal{C}$ be the set of all full chains and let $\mathcal{S}$ be the set of all spines. We say that a full chain $C \in \mathcal{C}$ is \emph{associated with a spine} $S \in \mathcal{S}$ or that {\em $C$ contains $S$ as a spine} if either 
\begin{enumerate}
	\item $C$ has exactly $k-1$ members of $\F$, which we name $F_1, \dots, F_{k-1}$. In this case, $C$ is associated with the spine that is a subchain of $C$ from $F_1$ to $F_{k-1}$; or 
	\item $C$ has exactly $k+x$ elements of $\F$ (where $x\geq 1$), which we name $F_1, \dots, F_{k+x}$. In this case, $C$ is associated with $x$ spines, namely $S_{F_{i}}$ for $2\leq i\leq x+1$, where $S_{F_i}$ is the spine that is a subchain of $C$ from $F_{i}$ to $F_{i+k-2}$.  (Notice that a chain~$C$ with at least $k+1$ elements of~$\F$ is \emph{not} associated with the spines that correspond to the first $k-1$ elements of~$\F \cap C$ and to the last $k-1$ elements of~$\F\cap C$.)
\end{enumerate}
Let $\spine(C)$ denote the set of all spines that $C$ contains. More precisely, $$\spine(C) := \{S : C \text{ contains }  S \text{ as a spine} \}.$$

\subsection{Overview of the discharging method} \label{overview}
In order to prove Theorem~\ref{thm:firstLYM} we use discharging arguments and Lemma~\ref{lem:spine_sum} below. We then prove Theorem~\ref{thm:conjresult} by using Theorem~\ref{thm:firstLYM} and induction on $k$. 

Before proving Lemma~\ref{lem:spine_sum}, we need the following straightforward counting lemma, the proof of which we provide for completeness.  
\begin{lem}\label{lem:moreempty}
	Let $n\geq 2$. 
	If $\mathcal{G} \subseteq \{\{1\}, \{1,2\}, \{1,2,3\}, \ldots, \{1,2,3.\ldots,n-1\}\}$, then the number of full chains in $2^{[n]}$ containing no member of $\mathcal{G}$ is at least the number of full chains that contain at least one member of $\mathcal{G}$.
\end{lem}

\begin{proof}
Let the set of chains that contain at least one member of $\mathcal{G}$ be $X$ and the set of chains that contain no member of $\mathcal{G}$ be $Y$. To show that $\abs X \le \abs Y$ we will construct an injection from $X$ to $Y$. Consider any chain $C \in X$. Let $C$ be $\varnothing \subset \{x_1\} \subset \{x_1,x_2\} \subset \{x_1,x_2,x_3\} \subset \ldots \subset \{x_1,x_2,x_3,\ldots,x_{n}\}$. For simplicity, we will say the permutation corresponding to $C$ is $x_1x_2x_3 \cdots x_{n}$. 

If $\{x_1,x_2,\ldots,x_j\}$ is the last set from $\mathcal{G}$ in $C$ and $x_i = 1$, then $x_1 x_2 \cdots x_j$ is a permutation of~$\{1,2,\ldots,j\}$. Hence, $x_{j+1} \geq j+1$ and $1 \le i \le j$. Let us consider the chain $C'$ corresponding to the permutation 
$$x_1 x_2 \cdots x_{i-1} x_{j+1} x_{i+1} x_{i+2} \cdots x_j x_i x_{j+2} \cdots  x_{n} , $$ 
obtained by swapping $x_{j+1} $ with $x_i$ in the permutation corresponding to $C$. If $C'$ contains the set $\{1,2,\ldots,j+1\}$, then it must be the case that $x_{j+1}=j+1$. Thus, $C$ contains the set $\{1,2,\ldots,j+1\}$, which contradicts the maximality of $j$. Therefore, under this map, the full chain $C'$ does not contain any member of $\mathcal{G}$.  If we map $C \in X$ to $C' \in Y$ in this way, the map is an injection, as desired.
\end{proof}


For the discharging step, we start by placing a weight on a spine depending on the chains that contain it. More precisely, if $S \in \mathcal{S}$ is a spine and $C \in \C$ is a full chain, then we define a weight function~$w(S, C)$ as follows. 
\begin{align*}
	w(S, C) =
	\begin{cases}
		1, &\text{ if  $S \in \spine(C)$ and $C$ contains at least $k+1$ members of $\F$,} \\
		-1, &\text{ if $S \in \spine(C)$ and $C$ contains exactly $k-1$ members of $\F$,} \\
		0, &\text{ otherwise}.
	\end{cases}
\end{align*}

Note that if $\F = \Sigma(n, k)$, then $\displaystyle\sum_{\substack{C \in \C \\ S \in\spine(C)}} w(S, C) = 0$.

\begin{lem}
	\label{lem:spine_sum}
	Let $k\geq 2$ be an integer and $n\geq k+1$. Let $\F$ be a family in $\B_n$ with no induced $Y_k$ and no induced $Y'_k$ such that $\varnothing$,~$[n]\not\in\F$. Let $\mathcal{S}$ denote the set of spines of $\F$ and let $\C$ denote the set of full chains in $\B_n$. For any $S \in \mathcal{S}$, $$\sum_{\substack{C \in \C \\ S \in\spine(C)}} w(S, C) \le 0.$$
\end{lem}
\begin{proof}
Let a spine $S$ be the chain $A_1\subset A_2\subset \cdots \subset A_{\ell}$ where $\abs{A_{i+1} \setminus A_{i}} = 1$  for $1\leq i\leq \ell-1$. (Recall that, by definition of a spine, there are exactly $k-1$ members of $\F$ in $\{ A_1, \dots, A_{\ell}\}$ and that $A_1$,~$A_{\ell} \in \F$.) If all the chains $C \in \C$ that contain $S$ as a spine have at most $k$ members of $\F$ then since $w(S, C) \in \{0,-1\}$ for each of these chains, our lemma follows trivially. 
Therefore, we may assume that there is a chain $C \in \C$ that contains $S$ as a spine and has at least $k+1$ members of $\F$; such a chain $C$ must have sets $P$,~$Q \in \F$ with $P \subset A_1$ and $A_{\ell} \subset Q$. 

If two sets $A$,~$B \in \F$ are unrelated to each other and $A$,~$B \subset A_1$ then we have an induced copy of $Y'_k$ consisting of $A$,~$B$, the $k-1$ members of $\F$ in $S$, and $Q$. Therefore, $\F \cap [\varnothing, A_1]$ induces a chain $\mathcal{G}_1$.  By symmetry, $\F \cap [A_{\ell}, [n]]$ induces a chain $\mathcal{G}_2$ as well. Since by assumption $\varnothing$,~$[n] \not \in \F$, the chains $\mathcal{G}_1\setminus \{A_1\}$ and $\mathcal{G}_2 \setminus \{A_{\ell}\}$ may be extended to chains that satisfy the hypotheses of Lemma \ref{lem:moreempty} for $[\varnothing, A_1]$ and $[A_{\ell}, [n]]$. 

Therefore, the number $a_0$ of full chains in $[\varnothing, A_1]$ containing no member of $\mathcal{G}_1\setminus \{A_1\}$ is at least the number $a_1$ of full chains in $[\varnothing, A_1]$ that contain a member of $\mathcal{G}_1\setminus \{A_1\}$. Similarly, the number $b_0$ of full chains in $[A_{\ell}, [n]]$ containing no member of $\mathcal{G}_2\setminus \{A_{\ell}\}$ is at least the number $b_1$ of full chains in $[A_{\ell}, [n]]$ that contain a member of $\mathcal{G}_2\setminus \{A_{\ell}\}$. Now notice that the number of chains $C \in \C$  associated with spine~$S$ that have exactly $k-1$ members of $\F$ is $a_0\cdot b_0$ and the number of chains $C \in \C$ associated with spine~$S$ that have at least $k+1$ members of $\F$ is $a_1\cdot b_1$. Therefore, since $a_1 \le a_0$ and $b_1 \le b_0$, 
\[
\sum_{\substack{C \in \C \\ \spine(C)\in S}} w(S, C) =  a_1\cdot b_1 - a_0\cdot b_0 \le 0. \qedhere
\]
\end{proof}

\section{Proofs of Theorem \ref{thm:conjresult} and Theorem \ref{thm:firstLYM}}\label{results}

First we use a folklore lemma that establishes an inequality very similar to the LYM inequality. 
A proof of this lemma occurs in~\cite{tompkinsthesis} as part of a proof of Erd\H{o}s' theorem.
Recall that $\Sigma(n,k)$ denotes the sum of the sizes of the largest $k$ levels in the Boolean lattice $2^{[n]}$.

\begin{lem}[See~{\cite[Lemma 1]{tompkinsthesis}}]\label{lem:kLYM}
	If $\F\subseteq 2^{[n]}$ satisfies 
	\begin{equation}\label{eq:kLYM}
		\sum_{F \in \F} \binom{n}{\abs F}^{-1}\leq k ,
	\end{equation}
	then $|\F|\leq\Sigma(n,k)$.
\end{lem}

\begin{pfct}{Theorem~\ref{thm:firstLYM}}  Observe that by Lemma~\ref{lem:spine_sum},
\begin{equation}\label{one}
	\sum_{S \in \mathcal{S}}\sum_{\substack{C \in \C \\ \spine(C)\ni S}} w(S, C) \le 0.
\end{equation}

Now notice that
\begin{equation}\label{two}
	\sum_{S \in \mathcal{S}}
	\sum_{\substack{C \in \C \\ \spine(C)\ni S}} 
	w(S, C) = 
	\sum_{C \in \C }
	\sum_{\substack{S \in \mathcal{S}  \\ S\in\spine(C)}} w(S, C)
\end{equation}
and that for any $C \in \C$, we have 
\begin{equation*}
	\sum_{\substack{S \in \mathcal{S}  \\ S\in\spine(C)}} w(S, C) = \abs{\F \cap C} - k.
\end{equation*}
Therefore, the right-hand side of \eqref{two} becomes 
\begin{equation}\label{four}
	\sum_{C \in \C } (\abs{\F \cap C} - k) = \sum_{F \in \F} \abs F! \cdot (n - \abs F)! - k\cdot n!.
\end{equation}

So by \eqref{one} and \eqref{four}, we have
\begin{equation*}
	\sum_{F \in \F} \abs F! \cdot (n - \abs F)! - k\cdot n! \le 0.
\end{equation*}
After rearranging, we obtain $\sum_{F \in \F} \binom{n}{\abs F}^{-1}\leq k$.
Lemma~\ref{lem:kLYM} gives that $\abs{\F}\leq\Sigma(n,k)$, proving Theorem~\ref{thm:firstLYM}.
\end{pfct}

\begin{pfct}{Theorem~\ref{thm:conjresult}}  The statement of Theorem~\ref{thm:conjresult} is true for $k=2$ (base case) due to Theorem~\ref{thm:zero}. 

If neither $\varnothing$ nor $[n]$ are in $\F$, then we may apply Theorem~\ref{thm:firstLYM} directly to obtain $\abs{\F}\leq\Sigma(n,k)$.

If both $\varnothing$ and $[n]$ are in $\F$, then $\F \setminus \{\varnothing,[n]\}$ is induced $Y_{k}$-free and induced $Y'_{k}$-free. Therefore, it has size at most $\Sigma(n,k-1)$ by the induction hypothesis. Since $2 + \Sigma(n,k-1) \le \Sigma(n,k)$ for $n \ge k+1$ and $k\ge 2$, we are done. 

Now, without loss of generality, suppose that $\varnothing \in \F$ and $[n] \not\in \F$. Now consider the family~$\F' := \F \setminus \{\varnothing\}$. By Theorem~\ref{thm:firstLYM}, we have 
\begin{equation}\label{equality_case}
	\sum_{F \in \F'} \binom{n}{\abs F}^{-1}\leq k
\end{equation}
and $\abs{\F'} \le \Sigma(n,k)$, by Lemma~\ref{lem:kLYM}. 

Now suppose $\abs{\F'} = \Sigma(n,k)$.  (Otherwise, we are done.)  A consequence of the proof of Lemma~\ref{lem:kLYM} is that, in order for equality to hold in~\eqref{eq:kLYM}, the quantities $\binom{n}{|F|}$ (for $F$ in $\F'$)  must be as large as possible---that is, the sets $F\in\F'$ must have size as close to $n/2$ as possible. More precisely, in order for equality to hold in~\eqref{eq:kLYM}, the list of the quantities $\binom{n}{|F|}$ for $F\in\F'$ in decreasing order (with multiplicities) must be the same as the list of the first $\Sigma(n,k)$ quantities $\binom{n}{|S|}$ for $S\subseteq 2^{[n]}$ in decreasing order (with multiplicities).

First, if $k$ and $n$ have different parities, then $\abs{\F'} = \Sigma(n,k)$ can only occur if 
\begin{align*}
	\F' = \binom{[n]}{\floor{\frac{n-k}{2}}} \cup \binom{[n]}{\floor{\frac{n-k}{2}}+1} \cup \cdots \cup \binom{[n]}{\floor{\frac{n-k}{2}}+k-1} .
\end{align*}
However, in that case, $Y_{k}$ is an induced subposet of $\F'$. Hence, adding $\varnothing$ produces an induced copy of $Y_k$ in $\F$, a contradiction.

Second, if $k$ and $n$ have the same parity, then $\abs{\F'} = \Sigma(n,k)$ can only occur if $\F'$ contains 
\begin{align*}
	\binom{[n]}{\frac{n-k}{2}+1} \cup \binom{[n]}{\frac{n-k}{2}+2} \cup \cdots \cup \binom{[n]}{\frac{n-k}{2}+k-1}
\end{align*}
plus $\binom{n}{\frac{n-k}{2}}$ sets from $\binom{[n]}{\frac{n-k}{2}} \cup \binom{[n]}{\frac{n-k}{2}+k}$.
If $\F'$ contains any set from $\binom{[n]}{\frac{n-k}{2}}$, then it is easy to see that $Y_{k}$ is an induced subposet of $\F'$ and adding $\varnothing$ produces an induced copy of $Y_k$ in $\F$. 
Otherwise, $\F'$ must contain all of the sets from $\binom{[n]}{\frac{n-k}{2}+k}$ and $n\geq k+2$. 
But in this case, $Y_{k}$ is again an induced subposet of $\F'$, giving an induced copy of $Y_k$ in $\F$, again a contradiction. 

Therefore, $\abs{\F'} \le \Sigma(n,k)-1$, which implies $\abs{\F} \le \Sigma(n,k)$, as desired.
\end{pfct}

\section{Concluding Remarks}

During the preparation of this article, we have learned that Tompkins and Wang recently proved Theorem \ref{thm:conjresult} independently \cite{TompkinsWang}. Their approach is closer to the method used in \cite{MCasey} and is different from the approach introduced in this article.

In fact, we believe that a more general result than Theorem~\ref{thm:conjresult} holds. Recall that $Y_{k,r}$ denotes the $r$-fork with a $k$-shaft poset and $Y'_{k,r}$ denotes its dual.
\begin{conj}\label{conj}
For all $k \ge 2$ and $r \ge 2$, there is an $n_0=n_0(k,r)$ such that if $n\ge n_0$, then $\La^{\sharp}(n, \{Y_{k,r}, Y'_{k,r}\}) = \Sigma(n,k)$.
\end{conj} 
Theorem~\ref{thm:conjresult} is the case when $r=2$; note that for all $k \geq 2$, $n_0(k,2)=k+1$.

The authors thank Kirk Boyer, Kaave Hosseini, Eric Sullivan and Casey Tompkins for many valuable discussions.

\end{document}